\documentclass[12pt]{amsart}
\advance\hoffset by -0.5 truecm
\usepackage{amssymb}
\usepackage{enumerate}
\newtheorem{Theorem}{Theorem}[section]
\newtheorem{Lemma}[Theorem]{Lemma}

\newtheorem{Proposition}[Theorem]{Proposition}

\newtheorem{Remark}[Theorem]{Remark}

\def\V{\mbox{Var}}

\def\R\re
\def\V{\bf V}

\def \re{{\mathbb R}}

\def \C{{\mathbb C}}

\def \V{{\bf V}}
\def \H{{\mathbb H}}

\def \cH{{\mathcal H}}

\begin{document}
\title[Helicity and the Ma\~n\'e critical value ]{Helicity and the Ma\~n\'e critical value}

\author[G.P. Paternain]{Gabriel P. Paternain}
 \address{ Department of Pure Mathematics and Mathematical Statistics,
University of Cambridge,
Cambridge CB3 0WB, UK}
 \email {g.p.paternain@dpmms.cam.ac.uk}

%\subjclass{53C25, 53C21, 58F17, 35J15}

%\date{February 2009}

%\maketitle

\begin{abstract}We establish a relationship between the helicity of a magnetic flow on a closed surface of genus $\geq 2$ and the Ma\~n\'e critical value.

\end{abstract}

\maketitle

\section{Results} Let $N$ be a closed oriented 3-manifold with a volume form $\Omega$. A vector field $F$ on $N$ that preserves $\Omega$ is said to
be {\it null-homologous} or {\it exact} if the closed 2-form $i_{F}\Omega$
is exact. Given a volume preserving null-homologous vector field $F$, the
{\it helicity} $\cH(F)$ is defined by setting
\[\cH(F):=\int_{N}\tau\wedge d\tau=\int_{N}\tau(F)\,\Omega,\]
where $\tau$ is any primitive of $i_{F}\Omega$. It is easy to check that
this definition is independent of the choice of primitive $\tau$.
The helicity (also referred to as the {\it asymptotic Hopf invariant}) measures
how much in average field lines wrap and coil around one another.
The term ``helicity'' was introduced by K. Moffatt \cite{M} who also pointed out
 the topological nature of the invariant. We refer to \cite{AK} for a complete account of this concept as well as its interpretation as an average self-linking number.

In this note we wish to study the following class of volume preserving flows.
Let $M$ be a closed oriented surface of genus $\geq 2$ and let $g$ be a Riemannian metric on $M$. The unit circle bundle $SM$ determined by $g$ is a closed
3-manifold with volume form $\Omega=\alpha\wedge d\alpha$, where $\alpha$ is
the contact 1-form of the geodesic flow of $g$. Let $X$ be the geodesic vector field and let $V$ denote the infinitesimal generator of the circle action on
the fibres of $SM$.
Suppose we are given in addition a 2-form $\sigma$ on $M$ (automatically closed). We may write $\sigma=f\,\mu_{g}$, where $\mu_g$ is the area form
of $g$. The vector field $F:=X+f\,V$ preserves the volume form
$\Omega$ and in fact is null-homologous. Its helicity is easy to compute
and one obtains:
\begin{equation}
\cH(F)=2\pi A+[\sigma]^{2}/\chi,
\label{hel}
\end{equation}
where $A$ is the area of $g$, $[\sigma]$ is the total integral of $\sigma$
and $\chi$ is the Euler characteristic of $M$ (we refer to Section \ref{proofs}
for proofs of these elementary facts). In fact formula (\ref{hel}) holds
also for the 2-sphere, but in that case $\cH(F)>0$ always.

There is a well known symplectic interpretation for the vector field
$F$. Denote by $\pi:TM\to M$ the canonical foot-point projection.
Consider the twisted symplectic form $\omega:=-d\alpha+\pi^*\sigma$
on $TM$ and the Hamiltonian $H:TM\to \re$ given by the kinetic energy
of $g$, i.e., $H(x,v):=g_{x}(v,v)/2$. The Hamiltonian vector field  $\xi_{H}$
of $H$ with respect to $\omega$ restricted to $SM$ is precisely
$F$ ($\xi_{H}$ is defined by $i_{\xi_{H}}\omega=dH$). 
Note that $SM$ coincides with the energy hypersurface $H=1/2$.
The flow of $F$ is called a {\it magnetic flow} or a {\it twisted geodesic flow}.
It is also well known, that considering $\xi_{H}$
restricted to a hypersurface $H=k$ is equivalent
to considering the vector field $F_{s}$ on $SM$ determined
by $s\,\sigma$, where $s=1/\sqrt{2k}$. We can think of $s$ as the
intensity that modulates the magnetic field $\sigma$. The projections
to $M$ of the orbits of $F_{s}$ model the motion of a particle with charge
$s$ under the effect of the magnetic field $\sigma$. If
we let $\omega_{s}:=-d\alpha+s\pi^*\sigma$, then the pair
$(SM,\omega_{s})$ determines a Hamiltonian structure for every $s$.

This note is partially motivated by the following two open questions:
\begin{enumerate}[A.] 
\item for which values of $s\in (0,\infty)$ is the Hamiltonian structure
 $(SM,\omega_{s})$ of contact type?
(Recall that contact type means that there exists a smooth 1-form $\lambda$
on $SM$ such that $d\lambda=\omega_s$ and $\lambda(F_{s})$ never vanishes.)
\item for which values of $s\in (0,\infty)$ does $F_s$ have a closed
orbit?
\end{enumerate}
The literature on Question B is vast and it is impossible to do it justice
in this brief note. Suffices to say that the study of the problem of existence of closed
orbits for magnetic flows was initiated by V.I. Arnold \cite{A} and S.P. Novikov \cite{N}
in the early 80's with subsequent work by many others. We refer the reader
to \cite{G1,G2} for a survey of some of these results, particularly for the 
case of surfaces discussed here.

Obviously A and B are related by Taubes' proof of the Weinstein
conjecture \cite{T}.
In the exact case $[\sigma]=0$, both questions were solved in \cite{CMP} with
the help of Aubry-Mather theory. An important ingredient in \cite{CMP}
was the Ma\~n\'e critical value whose definition we now recall.
Let $p:\tilde{M}\to M$ denote the universal covering of $M$. Since we are
assuming that $M$ has $\chi<0$, $\tilde{M}$ is diffeomorphic to $\re^2$
and thus $p^*\sigma$ has a primitive. Set
\[c(g,\sigma):= \inf_{\theta}\;\sup_{x\in \tilde{M}}\;
   \frac{1}{2}|\theta_{x}|^{2},\]
where the infimum runs over all 1-forms $\theta$ with $d\theta=p^*\sigma$,
and the norm of $\theta$ is taken with respect to the lifted Riemannian metric.
\footnote{More generally, there is a Ma\~n\'e critical value associated to
any covering of $M$ on which $\sigma$ becomes exact. The main result in \cite{CMP} asserts that if $\sigma$ is exact on $M$, a hypersurface with energy $k$ is of contact type iff $k>c_{0}$, where $c_0$ is the Ma\~n\'e critical value of the {\it abelian covering}.
This value coincides with the minimum of Mather's alpha function. Moreover,
every energy level has a closed orbit.}
It is easy to see that $c(g,\sigma)<\infty$ using the fact that on the
upper-half plane with the hyperbolic metric, the primitive $y^{-1}dx$ of the area form $y^{-2} dx\wedge dy$ is bounded. If $\sigma=\mu_g$, then $1/\sqrt{2c(g,\sigma)}$
coincides with Cheeger's isoperimetric constant of the universal covering
(cf. \cite{BP}). The critical value $c=c(g,\sigma)$
is also relevant for us because it is known that for
any $s\in (0,s_{c})$, every non-trivial homotopy class of $M$ contains
the projection of a closed orbit of $F_s$ \cite{P}, where $s_{c}:=1/\sqrt{2c}$.
A thorough discussion of the relevance of the Ma\~n\'e critical value
to the symplectic topology of hypersurfaces maybe found in \cite{CFP}.

Let us return to the helicity now. In the non-exact case ($[\sigma]\neq 0$)
an inspection of (\ref{hel}) tells us that
there is a unique positive value of $s$ for which $\cH(F_{s})=0$ and is given by
\[s_{h}^2:=\frac{-2\pi\chi\,A}{[\sigma]^2}.\]
Since the helicity vanishes, $(SM, \omega_{s_{h}})$ cannot be of
contact type. How does it relate with $s_c$?
The answer is given by the following theorem,

\medskip

\noindent {\bf Theorem.} {\it For an arbitrary pair $(g,\sigma)$ on a closed surface of genus $\geq 2$ with $[\sigma]\neq0$, we have
$s_{c}\leq s_{h}$ with equality if and only
if $g$ has constant Gaussian curvature and $\sigma$ is a constant multiple
of the area form of $g$.
}

\medskip

We note that if $g$ has constant curvature $-1$ and $\sigma=\mu_g$, then
$s_{c}=s_{h}=1$. The vector field $F_1$ is the horocycle flow, which
is uniquely ergodic and has of course, zero helicity. The theorem is saying that unless we are in this well understood homogeneous situation, if we wish
to answer Questions A and B above, we would need to wrestle with a non-trivial
interval $[s_{c},s_{h}]$ whose Hamiltonian structures are probably out of reach
of current technology. Presumably, every $s\in [s_{c},s_{h}]$ is {\it not}
of contact type, but even for $s_{c}$ this is not known in full generality.

The proof of the theorem has two ingredients. One was already present
in \cite[Theorem B]{BP} (but its relation with the helicity was not exposed)
and it will give fairly easily the inequality
and the fact that, if equality holds, then $g$ must
have constant curvature. However to show that $\sigma$ must be a constant
multiple of the area form requires a new tool. This is provided by techniques
closely related with the Selberg trace formula and the study of an appropriate
Radon transform of 1-forms on geodesic circles of the Riemann surface $M$.

\medskip

\noindent{\it Caveat on terminology.} In magnetohydrodynamics (MHD) the role
of $F$ is played by a magnetic field $\bf B$, frozen into a fluid of infinite conductivity filling $N$. What we call here magnetic field 
is the 2-form $\sigma$ and the two ``fields'' should not be confused.

\section{Proofs} 
\label{proofs}

Let $M$ be a closed oriented surface with genus $\geq 2$
 and $g$ a Riemannian
 metric. As above $SM$ is the unit circle bundle. We consider a 2-form
$\sigma$ on $M$ with total integral $[\sigma]$ and the magnetic flow on $SM$ defined by the pair
$(g,\sigma)$. The vector field of the magnetic flow is given
by $F:=X+fV$ where $f$ is defined by $\sigma=f\,\mu_{g}$, and
$\mu_{g}$ is the area form of $g$ with total integral $A$. There is a coframe of 1-forms 
$\{\alpha,\gamma,\psi\}$ in $SM$ related by the structure equations $d\alpha=\psi\wedge \gamma$, $d\gamma=-\psi\wedge \alpha$
and $d\psi=-K\,\alpha\wedge\gamma$, where $K$ is the Gaussian curvature of $g$ and $\alpha$
is the contact 1-form dual to the geodesic vector field $X$.
The coframe $\{\alpha,\gamma,\psi\}$ is dual to $\{X,H,V\}$, where $H=[V,X]$.

Let $\Omega:=\alpha\wedge d\alpha$ be the Sasaki volume form on $SM$.
A calculation using the structure equations shows that $i_{F}\Omega=d\alpha-f\alpha\wedge\gamma=d\alpha-f\pi^* \mu_g$. In other words $i_{F}\Omega=d\alpha-\pi^*\sigma$.
It is easy to find a primitive for $\pi^*\sigma$. Write
$\sigma=-aK\mu_{g}+d\beta$, where $a$ satisfies $[\sigma]=-a\,2\pi\chi$ and $\beta$ is a 1-form
on $M$.
Then $\pi^*\sigma=a\,d\psi+d\pi^*\beta$ and thus
$i_{F}\Omega=d\tau$, where $\tau:=\alpha-a\,\psi-\pi^*\beta$.
This shows that $F$ is null-homologous and
\[\tau(F)(x,v)=1-a\,f(x)-\beta_{x}(v).\]
Since the function $(x,v)\mapsto \beta_{x}(v)$ is odd with respect to
 the flip $v\mapsto -v$ we have
\[\int_{SM}\beta_{x}(v)\,\Omega=0.\]
It follows that the helicity of $F$ is given by
\[\cH(F)=2\pi A-a\,2\pi\,[\sigma]=2\pi A+[\sigma]^2/\chi,\]
which proves (\ref{hel}).

\begin{Remark}{\rm The calculation of the helicity for an arbitrary magnetic flow
was also carried out in \cite[Equation (4)]{P}. Up to the factor $-2\pi A$, the helicity
is precisely what I called in \cite{P} the action of the Liouville measure.
The Proposition in \cite{P} could be rephrased by saying that if $F_{s_{h}}$ has
no conjugate points, then $g$ must have constant curvature, $\sigma$ is a constant multiple
of the area form and $F_{s_{h}}$ is a horocycle flow. In fact, the proof of the
Proposition in \cite{P} shows that  for any $s\in (s_{h},\infty)$, $F_{s}$ has conjugate points.

The helicity for geodesic and horocycle flows in the constant curvature case
 is also computed in \cite[Proposition 4.9]{AK} and
\cite[Example 2.2.1]{VV}.

}
\end{Remark}

\begin{Remark}{\rm If we replace $\sigma$ by $s\sigma$ in the argument above
we obtain a primitive $\tau_{s}:=\alpha-as\,\psi-s\pi^*\beta$
of $i_{F_{s}}\Omega=-\omega_{s}$ such that
\[\tau_{s}(F_{s})=1-a\,s^2\,f(x)-s\,\beta_{x}(v).\]
This shows right away that if $f$ does not vanish (i.e. $\sigma$ is symplectic), then there is
$s_{0}$ such that for  any $s>s_0$, $(SM,\omega_{s})$ is of contact type
and therefore it has a closed orbit.

}
\end{Remark}

\subsection{Ma\~n\'e's critical value and helicity} By the conformal equivalence theorem
there exists a unique positive scalar $C^{\infty}$ function $\rho$ such that
the metric
$\rho^{2}g$ has constant negative curvature and the same area as $g$. 
Let $\rho_{g}$ be the {\it conformality coefficient} given by
\[\rho_{g}:=\frac{1}{A}\int_{M}\rho\,\mu_{g}.\]
By the Cauchy-Schwartz inequality, 
$\rho_{g}\leq 1$ and equality holds if and only if
$g$ itself is a metric of constant negative curvature.
In \cite[Theorem B]{BP} it is shown that
\begin{equation}
c(g,\sigma)\geq \frac{[\sigma]^2}{-4\pi \chi A\,\rho_{g}^2}.
\label{ineq:gk}
\end{equation}
In order to make this note self-contained we will give a proof of (\ref{ineq:gk}), which is actually
a little simpler than the one in \cite{BP} and it will naturally lead us to the proof of the Theorem. The key idea comes from a similar
estimate of Katok \cite{K} of the Cheeger isoperimetric constant.
Without loss of generality we may suppose that $g$ has area $A=-2\pi\chi$
and hence $g_{0}:=\rho^2\,g$ has constant negative curvature $-1$.

We lift everything to the universal covering $p:\H^2\to M$ and we consider the
Lagrangian $L(x,v)=g_{x}(v,v)/2-\theta_{x}(v)$, where $\theta$ is a primitive
of $p^{*}\sigma$. The critical value $c(g,\sigma)$ can also be characterised
as the infimum of the values of $k\in\re$ such that
the {\it action} 
\[A_{L+k}(\gamma):=\int_{0}^{T}(L+k)(\gamma(t),\dot{\gamma}(t))\,dt\geq 0\]
for all absolutely continuous closed curves $\gamma:[0,T]\to\H^2$
and any $T>0$. This was Ma\~n\'e's original approach to the critical
value \cite{Ma} and in the setting of non-exact magnetic fields, a proof
of the equivalence between these two ways of characterising $c$
may be found in \cite{BP}.

Consider a geodesic circle $C_r$ of $g_0$ 
of radius $r$. It has $g_0$-length
$2\pi\sinh r$ and encloses a disk $D_r$ of $g_{0}$-area $2\pi(\cosh r-1)$.
Its $g$-length is given by 
\[\ell_{g}(C_r)=\int_{0}^{2\pi\sinh r}\rho^{-1}(\gamma(t))\,dt,\]
where $\gamma:[0,2\pi\sinh r]\to\H^2$ is a parametrisation of $C_r$ 
with speed one with respect to $g_0$.
Now parametrise $C_r$ to have speed $\sqrt{2c}$ with respect to $g$.
We must have
\[A_{L+c}(C_{r})\geq 0\]
for all $r>0$.
In other words, using Stokes theorem and the definitions,
\[\sqrt{2c}\,\ell_{g}(C_r)-\int_{D_r}p^*\sigma\geq 0\]
for all $r>0$.
We can write $\sigma=a\,\mu_{g_{0}}+d\beta$, where $\mu_{g_{0}}$ is the area
form of $g_0$. Clearly $[\sigma]=-a\,2\pi\chi$. Thus
\begin{equation}
\sqrt{2c}\,\ell_{g}(C_r)-a\,2\pi(\cosh r-1)-\int_{C_r}p^*\beta\geq 0
\label{ineq:clave}
\end{equation}
for all $r>0$. 
The key observation now is that the projection to $M$ of a circle $C_r$ in
$\H^2$ converges to a horocycle when the radius goes to infinity,
and the projection to the unit sphere bundle of $(M,g_{0})$ of the
normalised arc length measure weakly converges to an invariant
probability measure for the horocycle flow. But the only invariant probability
measure for the horocycle flow is the Liouville measure $\nu$ of $g_0$ \cite{F}.

If we now divide (\ref{ineq:clave}) by $2\pi\sinh r$, let $r$ go to infinity and use
the definition of weak convergence
we derive 
\[\sqrt{2c}\int_{M}\rho^{-1}\,\mu_{g_{0}}-a\,A\geq 0,\]
since the integral of $\beta$ (regarded as a function on $TM$) over the unit circle
bundle of $g_0$ with respect to $\nu$ must vanish.
Equivalently, using that $[\sigma]=a\,A=-a\,2\pi\chi$ we have
\[\sqrt{2c}\geq \frac{[\sigma]}{\int_{M}\rho^{-1}\,\mu_{g_{0}}}.\]
But
\[-2\pi\chi\,\rho_{g}=\int_{M}\rho\,\mu_{g}=\int_{M}\rho^{-1}\,\mu_{g_{0}}.\]
Hence
\[\sqrt{2c}\geq  \frac{[\sigma]}{-2\pi\chi\,\rho_{g}}\]
which proves (\ref{ineq:gk}) when $A=-2\pi\chi$ (note that
$c(g,\sigma)=c(g,-\sigma)$).

Next we observe that the inequality $s_{c}^2\leq s_{h}^2$ is
equivalent to $c\geq \frac{[\sigma]^2}{-4\pi \chi A}$ which follows
immediately from (\ref{ineq:gk}) since $\rho_{g}\leq 1$. Moreover, if equality
holds then $\rho_{g}=1$ and $g$ must have constant curvature.
What remains to prove in the Theorem from the introduction
is that if $s_c=s_h$ then $f$ must also
be constant. Before we proceed any further we would like to record
inequality (\ref{ineq:gk}) in the following form,

\begin{Proposition} For any pair $(g,\sigma)$ on a surface with genus $\geq 2$
we have
\[2\,c(g,\sigma)\geq 2\rho_{g}^2\,c(g,\sigma)\geq 1-\frac{\cH(F)}{2\pi A}.\]
\end{Proposition}

As above, and without loss of generality we shall
assume that $g$ has constant curvature $-1$. Since $s_{c}=s_{h}$
we can write $\sigma=a\mu_{g}+d\beta$, where $[\sigma]=-a\,2\pi\chi$ 
and $2c=a^2$. We may suppose that in fact $a=\sqrt{2c}$.
Using (\ref{ineq:clave}) we obtain

\[2\pi\,\sqrt{2c}\,(1+\sinh r-\cosh r)-\int_{C_r}p^*\beta\geq 0\]
which implies
\begin{equation}
\int_{C_r}p^*\beta\leq 2\pi\,\sqrt{2c}\,(1-e^{-r})\leq 2\pi\,\sqrt{2c}
\label{ineq:rad}
\end{equation}
for all $r>0$. In the next subsection we explain how use the bound (\ref{ineq:rad}) to show that in fact $\beta$ must be closed, and consequently $\sigma$
is a constant multiple of $\mu_g=\mu_{g_{0}}$.

\subsection{A Radon transform}
Let $h:M\to \re$ be a smooth function with $\int_{M}h(x)\,dx=0$.
As before $M=\H^2/\Gamma$ and $p:\H^2\to M$ the quotient map.
 We consider the Radon transform $\hat{h}_r$
of $h$ on geodesic disks defined as follows. 
Given $x\in M$, let $\tilde{x}$ be a lift of $x$ and let
$D(\tilde{x},r)$ be the disk with center $\tilde{x}$ and radius $r$.
We set
\[\hat{h}_{r}(x):=\int_{D(\tilde{x},r)}h\circ p(y)\,dy.\]
It is easy to check that this definition is independent of the
lift of $x$.

Let $\varphi_{0},\varphi_{1},\cdots$ denote a complete orthonormal
sequence of {\it real} eigenfunctions of the Laplacian of $M$ corresponding
to eigenvalues $0=\lambda_{0}<\lambda_{1}\leq \lambda_{2}\leq \cdots\nearrow \infty$. Write $h=\sum_{j} a_{j}\varphi_{j}$. Since $h$ has zero average
over $M$, $a_{0}=0$.

\begin{Lemma}$\hat{h}_{r}(x)=\sum_{j} a_{j}q_{r}(s_{j})\varphi_{j}(x)$, where
$q_{r}$ is the function ($s\in\C$):
\[q_{r}(s)=4\sqrt{2}\int_{0}^{r}\cos su\,(\cosh r-\cosh u)^{1/2}\,du\]
and $s_{j}$ is any of the roots of $\frac{1}{4}+s_{j}^2=\lambda_{j}$.

\end{Lemma}

\begin{proof} This is the application of the techniques connected
with the Selberg trace formula
\cite{S} and is fully explained in Section 2 of Randol's chapter
in Chavel's book \cite{C}. It goes as follows. We let $k_{r}(x,y)$ be 
the function on $\H^2\times \H^2$ such that $k_{r}(x,y)=1$ if $y\in D(x,r)$
and $k_{r}(x,y)=0$ otherwise. Set
$K_{r}(x,y):=\sum_{\gamma\in\Gamma}k_{r}(x,\gamma y)$. Then it is easy to check
that
\[\hat{h}_{r}(x)=\int_{M}h(y)K_{r}(x,y)\,dy.\]
Using the expansion $h=\sum_{j} a_{j}\varphi_{j}$ we obtain
\[\hat{h}_{r}(x)=\sum_{j}a_{j}\int_{M}\varphi_{j}(y)K_{r}(x,y)\,dy.\]
But it is shown in \cite[Chapter X, Theorem 1 and p. 277]{C} that
\[\int_{M}\varphi_{j}(y)K_{r}(x,y)\,dy=q_{r}(s_{j})\varphi_{j}(x),\]
where $q_{r}(s)$ is calculated in \cite[p. 275]{C} yielding the formula
in the lemma.

\end{proof}

\begin{Remark} {\rm If $s_{j}$ is real, then  
\[q_{r}(s_j)=4\sqrt{2}\int_{0}^{r}\cos s_j u\,(\cosh r-\cosh u)^{1/2}\,du.\]
If there are small eigenvalues, then $s_j$ would be purely imaginary
and if we let $\alpha_j=|s_{j}|$, then
\[q_{r}(s_j)=4\sqrt{2}\int_{0}^{r}\cosh \alpha_{j} u\,(\cosh r-\cosh u)^{1/2}\,du.\]
}
\label{remark:smalle}

\end{Remark}

\begin{Lemma} For every $j$, there is $r_{n}\to\infty$ such that
$q_{r_{n}}(s_{j})\to\infty$.
\label{lemma:q}
\end{Lemma}

\begin{proof} On account of Remark \ref{remark:smalle} it suffices to prove
 the lemma when $s_j$ is real and positive since if $s_j$ is purely imaginary
or zero $\cosh\alpha_j u\geq 1$ for all $u$ which implies
that $q_{r}(s_{j})\geq q_{r}(s)$ where $s$ is any real number.

Suppose then that $s_j$ is real and positive and note
\begin{align*}
\frac{q_{r}(s_{j})}{4\sqrt{2}}&
=\int_{0}^{r}\frac{\cos s_{j}u\,(\cosh r-\cosh u)}{(\cosh r-\cosh u)^{1/2}}\,du\\
&\geq \frac{1}{(\cosh r) ^{1/2}}\int_{0}^{r}\cos s_{j}u\,(\cosh r-\cosh u)\,du\\
&=\frac{1}{(\cosh r)^{1/2}\,s_{j}(1+s_{j}^2)}(\cosh r\sin s_{j}r-s_{j}\sinh r\cos s_{j}r).
\end{align*}
Thus, if we take $r_{n}=\frac{\pi (2n+1/2)}{s_{j}}$ we derive
\[\frac{q_{r_{n}}(s_{j})}{4\sqrt{2}}\geq \frac{(\cosh r_{n})^{1/2}}{s_{j}(1+s_{j}^{2})}\]
which proves the lemma.

\end{proof}

\begin{Lemma} Suppose that $\hat{h}_{r}(x)\leq C$ for all $x\in M$ and
all $r>0$. Then $h\equiv 0$.
\label{lemma:zero}
\end{Lemma}

\begin{proof} We will show that all the Fourier coefficients $a_{j}$ vanish.
Suppose $a_{k}\geq 0$ for some $k$.
By compactness there is a constant $B_k$ such that $\varphi_k+B_k$
is a positive function. Multiply both sides of the inequality
\[\sum_{j} a_{j}q_{r}(s_{j})\varphi_{j}(x)\leq C\]
by $\varphi_k(x)+B_k$ and integrate with respect to $x$ to obtain ($a_{0}=0$)
\[a_{k}q_{r}(s_{k})\leq C\,B_{k}\,A\]
for all $r>0$. By Lemma \ref{lemma:q} this can only happen if $a_{k}=0$.

If $a_{k}\leq 0$ we proceed in a similar way by considering a constant
$B_k$ such that $\varphi_k+B_k$ is a negative function. In any case we obtain
$a_k=0$ as desired.

\end{proof}

\subsection{End of the proof of the Theorem} Write $d\beta=h\,\mu_{g}$ where
$h$ has zero average over $M$. Inequality (\ref{ineq:rad})
is saying that $\hat{h}_{r}(x)\leq 2\pi\,\sqrt{2c}$ for all $x\in M$ and
$r>0$. By Lemma \ref{lemma:zero}, $h$ vanishes identically and $\beta$
must be closed.

\end{document}